\documentclass[a4paper, 12pt]{article}
\usepackage{mathrsfs}
\usepackage{amsmath}
\usepackage{amsfonts}
\usepackage[T1]{fontenc}
\usepackage[latin9]{inputenc}
\usepackage{amsthm}
\usepackage{graphicx}
\usepackage{amsmath}

\usepackage{amsthm}
\usepackage{array}
\usepackage{cases}
\makeatletter
\def\th@plain{%
  \itshape 
}
\makeatother

\makeatletter
\renewenvironment{proof}[1][\proofname]{\par
  \pushQED{\qed}%
  \normalfont \topsep6\p@\@plus6\p@\relax
  \trivlist
  \item[\hskip\labelsep
        \bfseries
    #1\@addpunct{.}]\ignorespaces
}{%
  \popQED\endtrivlist\@endpefalse
}
\makeatother

\numberwithin{equation}{section}

\newtheorem{thm}{Theorem}[section]

\newtheorem{lem}[thm]{Lemma}

\newtheorem{conj}[thm]{Conjecture}

\numberwithin{equation}{section}

\usepackage{graphicx, epsfig, subfigure}
\usepackage{enumerate} 
\usepackage[square, numbers, sort&compress]{natbib} 
\ifx\pdfoutput\undefined
 \usepackage[dvipdfm,%
  pdfstartview=FitH, 
  bookmarks=true,%
  bookmarksnumbered=true, 
  bookmarksopen=true, 
  plainpages=false,%
  pdfpagelabels,%
  colorlinks=true, 
  linkcolor=blue, 
  citecolor=blue,%
  urlcolor=black,
  pdfborder=001]{hyperref}
  \AtBeginDvi{}  
\else
 \usepackage[pdftex,%
  pdfstartview=FitH, 
  bookmarks=true,%
  bookmarksnumbered=true, 
  bookmarksopen=true, 
  plainpages=false,%
  pdfpagelabels,%
  colorlinks=true, 
  linkcolor=blue, 
  citecolor=blue,%
  urlcolor=black,
  pdfborder=001]{hyperref}
\fi

\numberwithin{equation}{section}

\begin{document}

\title{\LARGE Equitable partition of graphs into induced linear forests
}
\author{Xin Zhang\,\thanks{Supported by the National Natural Science Foundation of China (No.\,11871055).}~\thanks{Corresponding author. Emails: xzhang@xidian.edu.cn (X. Zhang) beiniu@stu.xidian.edu.cn (B. Niu).}\,\, \,\,\,\,Bei Niu\\
{\small School of Mathematics and Statistics, Xidian University, Xi'an, Shaanxi, 710071, China}\\
}

\maketitle

\begin{abstract}\baselineskip 0.60cm
It is proved that the vertex set of any simple graph $G$ can be equitably partitioned into $k$ subsets for any integer $k\geq\max\{\big\lceil\frac{\Delta(G)+1}{2}\big\rceil,\big\lceil\frac{|G|}{4}\big\rceil\}$ so that each of them induces a linear forest.
\\\vspace{3mm}\noindent \emph{Keywords: equitable coloring; vertex arboricity; linear forest}.
\end{abstract}

\baselineskip 0.60cm

\section{Introduction }

All graphs considered in this paper are simple and finite. A \emph{tree- (resp.\,path-) $k$-coloring} of a graph $G$ is a function $c$ from $V(G)$ to the set $\{1,2,\dots,k\}$ so that $c^{-1}(i)$, the \emph{color class} $i$, induces a forest (resp.\,linear forest) for each integer $1\leq i\leq k$. Here a \emph{linear forest} is a forest with each connected component being a path.

A tree- (resp.\,path-) $k$-coloring is \emph{equitable} if the sizes of any two color classes differ by at most one. The minimum integer $k$ such that a graph $G$ admits an equitable tree- (resp.\,path-) $k$-coloring  is the \emph{equitable vertex arboricity} (resp.\,\emph{equitable linear vertex arboricity}) of $G$, denoted by $va^{=}(G)$ (resp.\,$lva^{=}(G)$). Note that the complete bipartite graph $K_{9,9}$ has equitable vertex arboricity (resp.\,equitable linear vertex arboricity) two, but it is impossible to construct an equitable tree- (resp.\,path-) $3$-coloring  of $K_{9,9}$. This motivates us to define another chromatic parameter so-called the \emph{equitable vertex arborable threshold} (resp.\,\emph{equitable linear vertex arborable threshold}). Formally, it is the minimum integer $k$ such that $G$ admits an equitable tree- (resp.\,path-) $k'$-coloring for every integer $k'\geq k$, denoted by $va^{\equiv}(G)$ (resp.\,$lva^{\equiv}(G)$). Clearly, $va^{=}(G)\leq va^{\equiv}(G)$ and  $lva^{=}(G)\leq lva^{\equiv}(G)$.

For the complete bipartite graph $K_{n,n}$, it is trivial that $va^{=}(K_{n,n})=2$. For its equitable vertex arborable threshold, Wu, Zhang and Li \cite{WZL.2013} showed that $va^{\equiv}(K_{n,n})=2\big\lfloor(\sqrt{8n+9}-1)/4\big\rfloor$ if $2n=t(t+3)$ and $t$ is odd. This implies that the gap between $va^{=}(G)$ and $va^{\equiv}(G)$ can be any large.
Since $2=lva^{=}(K_{n,n})=va^{=}(K_{n,n})\leq va^{\equiv}(K_{n,n})\leq lva^{\equiv}(K_{n,n})$, the gap between $lva^{=}(G)$ and $lva^{\equiv}(G)$ can also be any large.

The notions of the equitable vertex arboricity and the equitable vertex arborable threshold were introduced by Wu, Zhang and Li \cite{WZL.2013} in 2013, who put forward the following two conjectures.

\begin{conj}[\emph{\textbf{Equitable Vertex Arboricity Conjecture}}]\label{conj1}
$va^{\equiv}(G)\leq \big\lceil\frac{\Delta(G)+1}{2}\big\rceil$ for every graph $G$.
\end{conj}

\begin{conj}\label{conj2}
There is a constant $C$ such that $va^{\equiv}(G)\leq C$ for every planar graph $G$.
\end{conj}

In 2015, Esperet, Lemoine and Maffray \cite{ELM.2015} confirmed Conjecture \ref{conj2} by showing that
$va^{\equiv}(G)\leq 4$ for every planar graph $G$. Recently, Niu, Zhang and Gao \cite{NZG} proved that $va^{\equiv}(G)\leq 8$ for every IC-planar graph $G$ (a graph is \emph{IC-planar} if it has embedding in the plane so that each edge is crossed by at most one other edge and each vertex is incident with at most one crossing edge).

For Conjecture \ref{conj1}, it is still widely open, and there are some partial results in the literature. For example, Zhang \cite{Z.2016} verified it for subcubic graphs and Chen et al.\,\cite{CGSWW.2017} confirmed it for 5-degenerate graphs.

In many papers, including \cite{CGSWW.2017,Z.2015,Z.2016}, the authors announced that Conjecture \ref{conj1} has been confirmed for graphs $G$ with $\Delta(G)\geq |G|/2$ by Zhang and Wu \cite{ZW.2014}. However, one can look into that paper and then find that Zhang and Wu just proved a weaker result that $va^{=}(G)\leq \big\lceil(\Delta(G)+1)/2\big\rceil$ for every graph $G$ with $\Delta(G)\geq |G|/2$, and their result (even their proof) cannot implies $va^{\equiv}(G)\leq \big\lceil(\Delta(G)+1)/2\big\rceil$ for such a graph $G$. This motivates us to write this paper to give a detailed proof of the following theorem, which confirms Conjecture \ref{conj1} for graphs $G$ with $\Delta(G)\geq (|G|-1)/2$.

\begin{thm}\label{EVAT}
If $G$ is a graph with $\Delta(G)\geq\frac{|G|-1}{2}$ and $k\geq\big\lceil\frac{\Delta(G)+1}{2}\big\rceil$ is an integer, then $V(G)$ can be equitably partitioned into $k$ subsets so that each of them induces a linear forest.
\end{thm}

Actually, Theorem \ref{EVAT} implies the following

\begin{thm}\label{ELVA}
$lva^{\equiv}(G)\leq \big\lceil\frac{\Delta(G)+1}{2}\big\rceil$ for graphs $G$ with $\Delta(G)\geq\frac{|G|-1}{2}$.
\end{thm}

Since the complete graph $K_n$ satisfies that $\Delta(K_n)=n-1\geq |K_n|/2$ and $lva^{\equiv}(K_n)= \big\lceil n/2\big\rceil=\big\lceil(\Delta(G)+1)/2\big\rceil$, the lower bound for $k$ in Theorem \ref{EVAT} and the upper bound for $lva^{\equiv}(G)$ in Theorem \ref{ELVA} are sharp in this sense.

The proof of Theorem \ref{EVAT} will be given in Section \ref{sec:2}. In Section \ref{sec:3} , we will give a slightly stronger result that omits the condition $\Delta(G)\geq (|G|-1)/2$ in Theorem \ref{ELVA} but replaces the upper bound for $lva^{\equiv}(G)$ with $\max\{\big\lceil(\Delta(G)+1)/2\big\rceil,\big\lceil |G|/4 \big\rceil\}$.

\vspace{3mm}\noindent \textbf{Notations:} we use standard notations that come from the book on Graph Theory contributed by Bondy and Murty \cite{Bondy.2008}. In the next section there are two notations $\alpha'(G)$ and $G^c$ that are frequently used.  They respectively denote the largest size of the matching in the graph $G$ and the completement graph of $G$. 

\section{A constructive proof of Theorem \ref{EVAT}}\label{sec:2}

In order to give the proof of Theorem \ref{EVAT}, we collect some useful lemmas concerning the structure of a graph. For convenience, we list them here in advance.

\begin{lem}
\label{lem-exe}
If $G$ is a connected graph with minimum degree $\delta\leq \frac{|G|-1}{2}$, then $G$ contains a path of length $2\delta$.
\end{lem}

\begin{proof}
Let $P=x_0x_1\cdots x_k$ be the longest path of $G$. It is sufficient to prove that $k\geq 2\delta$ and thus the required path is contained in $P$. Suppose, to the contrary, that $k\leq 2\delta-1$. Since $P$ is the longest path, the neighbors of $x_0$ or $x_k$ are all on $P$. Let $S=\{i~|~x_0x_{i+1}\in E(G), 0\leq i\leq k-1\}$ and let $T=\{i~|~x_ix_k\in E(G), 0\leq i\leq k-1\}$.
It is clear that $2\delta\leq d_G(x_0)+d_G(x_k)=|S|+|T|=|S\cup T|+|S\cap T|\leq k+|S\cap T|$, which implies that $|S\cap T|\geq 2\delta-k\geq 1$. Suppose
$j\in S\cap T$. It follows that $x_0x_{j+1},x_jx_k\in E(G)$ and thus there is a cycle $C$ on $k+1$ vertices, say $x_0x_{j+1}x_{j+2}\cdots x_kx_jx_{j-1}\cdots x_1x_0$. Since $G$ is connected and $|G|\geq 2\delta+1\geq k+2$, outside the cycle $C$ there is a vertex $y$ that connects to some vertex $x_r$ of $C$, where $0\leq r\leq k$. In this case, one can immediately find a path on $k+2$ vertices from the graph induced by $E(C)\cup \{yx_r\}$, contradicting the assumption that $P$ is the longest path in $G$.
\end{proof}

\begin{lem}
\label{lem2}
If $G$ is a connected graph such that $|G|>2\delta(G)$, then $\alpha'(G)\geq\delta(G)$.
\end{lem}
\begin{proof}
By Lemma \ref{lem-exe}, $G$ contains a path $P=x_0x_1\cdots x_{2\delta(G)}$ of length $2\delta(G)$. Hence there exists a matching $\{x_0x_1,x_2x_3,\cdots,x_{2\delta(G)-2}x_{2\delta(G)-1}\}$ of size $\delta(G)$, which implies that $\alpha'(G)\geq\delta(G)$.
\end{proof}

\begin{lem}
\label{lem3}
If $G$ is a graph with $\delta(G)\geq2$, then $G$ contains a cycle of length at least $\delta(G)+1$.
\end{lem}

\begin{proof}
Let $P=x_0x_1\cdots x_k$ be the longest path of $G$. It is clear that all neighbors of $x_0$ are on $P$.  Let $x_i$ be a neighbor of $x_0$ so that $i$ is maximum (actually $i$ is exactly the degree of $v_0$ in $G$, and thus is at least $\delta(G)$). Since $\delta(G)\geq 2$, $C=x_0x_1\ldots x_ix_0$ is a cycle of length $i+1\geq \delta(G)+1$, as required.
\end{proof}

\begin{lem}
\label{lem1}
If $G$ is a disconnected graph, then $\alpha'(G)\geq\delta(G)$.
\end{lem}
\begin{proof}
If $\delta(G)\leq 1$, then there is nothing to prove. Hence we assume $\delta(G)\geq 2$.
Let $G_1$ and $G_2$ be two components of $G$. It follows that $\delta(G_1),\delta(G_2)\geq \delta(G)\geq 2$. By Lemma \ref{lem3}, $G_1$ or $G_2$ contains a cycle $C_1=x_0x_1\cdots x_rx_0$ or $C_2=y_0y_1\cdots y_sy_0$ with $r\geq \delta(G)$ or $s\geq \delta(G)$, respectively. Under this condition, we can construct a matching
$$\{x_0x_1,x_2x_3,\cdots,x_{2\lfloor(\delta(G)+1)/2\rfloor-2}x_{2\lfloor(\delta(G)+1)/2\rfloor-1},
y_0y_1,y_2y_3,\cdots,y_{2\lfloor\delta(G)/2\rfloor-2}y_{2\lfloor\delta(G)/2\rfloor-1}\}$$
of size $\lfloor\frac{\delta(G)+1}{2}\rfloor+\lfloor\frac{\delta(G)}{2}\rfloor=\delta(G)$, which implies $\alpha'(G)\geq\delta(G)$.
\end{proof}

Combining Lemma \ref{lem-exe} with Lemma \ref{lem3}, we immediately have the following

\begin{lem}\label{lem-use}
If $G$ is a graph with $2\leq \delta(G)\leq \frac{|G|-1}{2}$, then $G$ contains two vertex-disjoint paths $P_1$ and $P_2$ such that $|P_1|=\delta(G)+1$ and $|P_2|=\delta(G)$.
\end{lem}

\begin{proof}
If $G$ is connected, then by Lemma \ref{lem-exe}, $G$ contains a path of length $2\delta(G)$, which can be split  into the required two vertex-disjoint paths. If $G$ is disconnected, then $G$ contains at least two components $G_1$ and $G_2$, and the minimum degree of $G_1$ and $G_2$ are both at least $\delta(G)\geq 2$.
By Lemma \ref{lem3}, there are cycles $C_1\subseteq G_1$ and $C_2\subseteq G_2$ of length at least $\delta(G)+1$. Clearly, we can choose $P_1\subseteq C_1$ and $P_2\subseteq C_2$ such that $|P_1|=\delta(G)+1$ and $|P_2|=\delta(G)$, as required.
\end{proof}

We are ready to prove Theorem \ref{EVAT}. Note that $V(G)$ can be equitably partitioned into $k$ subsets if and only if $V(G)$ can be partitioned into $k$ subsets  so that each subset contains either $\big\lfloor\frac{|G|}{k}\big\rfloor$ or $\big\lceil\frac{|G|}{k}\big\rceil$ vertices. We spit the proof into three parts according to the value of $k$.

\vspace{3mm}\textbf{Case 1.}  $k\geq\frac{|G|}{2}$.\vspace{3mm}

In this case, we have $$\bigg\lceil\frac{|G|}{k}\bigg\rceil\leq 2.$$ Hence we arbitrarily partition $V(G)$ into $k$ subsets so that each subset consists of one or two vertices (and thus induces a linear forest), as required.

\vspace{3mm}\textbf{Case 2.}  $\frac{|G|}{3}\leq k<\frac{|G|}{2}$.\vspace{3mm}

In this case, we have
$$2\leq \bigg\lfloor\frac{|G|}{k}\bigg\rfloor\leq \bigg\lceil\frac{|G|}{k}\bigg\rceil\leq 3.$$
In the following, we partition
$V(G)$ into $k$ subsets so that each subset contains two or three vertices.

Using $\Delta(G)+\delta(G^{c})=|G|-1$ and $\Delta(G)\geq\frac{|G|-1}{2}$, we deduce that $|G|=|G^{c}|\geq2\delta(G^{c})+1$. According to Lemmas \ref{lem2} and \ref{lem1}, we immediately have $\alpha':=\alpha'(G^{c})\geq\delta(G^{c})$, which implies the existence of a matching $M=\{x_{1}y_{1},\cdots,x_{\alpha'}y_{\alpha'}\}$ in $G^{c}$.

Since $k\geq\big\lceil\frac{\Delta(G)+1}{2}\big\rceil$, $$\alpha'\geq \delta(G^c)=|G|-(\Delta(G)+1)\geq |G|-2k.$$ Hence we can obtain a subset $M'=\{x_{1}y_{1},\cdots,x_{|G|-2k}y_{|G|-2k}\}$ of $M$. Let $z_1,z_2,\cdots,z_{|G|-2k}$ be distinct vertices in $V(G)\backslash  V(M')$ and let $U_i=\{x_i,y_i,z_i\}$ with $1\leq i\leq |G|-2k$. Clearly, each $U_i$ induces a linear forest in $G$. Since $|V(G)\backslash \bigcup_{i=1}^{|G|-2k}U_i|=|G|-3(|G|-2k)=6k-2|G|\geq 0$, we arbitrarily partition $V(G)\backslash \bigcup_{i=1}^{|G|-2k}U_i$ into $3k-|G|$ disjoint subsets $W_1,W_2,\cdots,W_{3k-|G|}$ so that each of them contains exactly two vertices. Note that  each $W_i$ induces a linear forest in $G$. Hence $$U_1,U_2,\cdots,U_{|G|-2k},W_1,W_2,\cdots,W_{3k-|G|}$$ is the desired partition of $V(G)$.

\vspace{3mm}\textbf{Case 3.} $\lceil\frac{\Delta(G)+1}{2}\rceil\leq k<\frac{|G|}{3}$\vspace{3mm}

In this case, we have
\begin{align}\label{eq1}
3\leq \bigg\lfloor\frac{|G|}{k}\bigg\rfloor\leq \bigg\lceil\frac{|G|}{k}\bigg\rceil   \leq \Bigg\lceil\frac{|G|}{\lceil\frac{\Delta(G)+1}{2}\rceil}\Bigg\rceil
\leq \Bigg\lceil\frac{|G|}{\frac{|G|+1}{4}}\Bigg\rceil      = 4.
\end{align}
Moreover, we have
\begin{align}\label{eq2}
|G|\leq 4k-1.
\end{align}
If not, then $|G|=4k$ by \eqref{eq1} and thus we have $\Delta(G)\geq\lceil\frac{|G|-1}{2}\rceil=2k$ (note that $\Delta(G)$ shall be an integer), which implies $|G|-(\Delta(G)+1)\leq 2k-1$. However, we have, on the other hand, that $|G|-(\Delta(G)+1)\geq |G|-2k=2k$, since $k\geq \lceil\frac{\Delta(G)+1}{2}\rceil\geq \frac{\Delta(G)+1}{2}$. This results in a contradiction.

In the following, we are to partition
$V(G)$ into $k$ subsets so that each subset contains three or four vertices.
Since $\Delta(G)+\delta(G^{c})=|G|-1$ and $\Delta(G)\geq\frac{|G|-1}{2}$,  $|G|=|G^{c}|\geq2\delta(G^{c})+1$. By Lemma \ref{lem-use}, $G^c$ contains two vertex-disjoint paths $P_1=x_0x_1\cdots x_{\delta}$ and $P_2=y_0y_1\cdots y_{\delta-1}$, where $\delta:=\delta(G^c)$.

Let $\beta=|G|-3k$ and $\mu=4k-|G|$. By \eqref{eq1} and \eqref{eq2}, $\beta,\mu\geq 1$. Since
$|G|-2k\leq |G|-(\Delta(G)+1)=\delta$, we conclude
\begin{align}\label{eq3}
2\beta+1\leq 2\beta+\mu\leq \delta.
\end{align}

Let $\rho=2\lceil\frac{\beta}{2}\rceil-\beta$ and let
\begin{align}
  \label{V1i} V^1_i&=\{x_{4i-4},x_{4i-3},x_{4i-2},x_{4i-1}\}, 1\leq i\leq \bigg\lceil\frac{\beta}{2}\bigg\rceil\\
  \label{U1i} U^1_i&=\{x_{2i},x_{2i+1}\}, 2\bigg\lceil\frac{\beta}{2}\bigg\rceil\leq i\leq 2\bigg\lceil\frac{\beta}{2}\bigg\rceil+\bigg\lfloor\frac{\mu+1}{2}\bigg\rfloor-\rho-1\\
  \label{V2i} V^2_i&=\{y_{4i-4},y_{4i-3},y_{4i-2},y_{4i-1}\}, 1\leq i\leq \bigg\lfloor\frac{\beta}{2}\bigg\rfloor\\
  \label{U2i} U^2_i&=\{y_{2i},y_{2i+1}\}, 2\bigg\lfloor\frac{\beta}{2}\bigg\rfloor\leq i\leq 2\bigg\lfloor\frac{\beta}{2}\bigg\rfloor+\bigg\lfloor\frac{\mu}{2}\bigg\rfloor+\rho-1
\end{align}
Note that $0\leq \rho\leq 1$ and the upper bound for $i$ in \eqref{U1i} or \eqref{U2i} may be less than its lower bound, in which case we naturally ignore the definition of $U^1_i$ or $U^2_i$, and also the definition of $W^1_i$ or $W^2_i$ that will be introduced later.

Since
$$4\bigg\lceil\frac{\beta}{2}\bigg\rceil-1\leq 4\cdot \frac{\beta+1}{2}-1=2\beta+1\leq \delta$$
$$2\bigg(2\bigg\lceil\frac{\beta}{2}\bigg\rceil+\bigg\lfloor\frac{\mu+1}{2}\bigg\rfloor-\rho-1\bigg)+1= 2\bigg(\bigg\lfloor\frac{\mu+1}{2}\bigg\rfloor+\beta-1\bigg)+1\leq 2\bigg(\frac{\mu+1}{2}+\beta-1\bigg)+1=2\beta+\mu\leq \delta$$
$$4\bigg\lfloor\frac{\beta}{2}\bigg\rfloor-1\leq 4\cdot \frac{\beta}{2}-1=2\beta-1\leq \delta-2< \delta-1$$
$$2\bigg(2\bigg\lfloor\frac{\beta}{2}\bigg\rfloor+\bigg\lfloor\frac{\mu}{2}\bigg\rfloor+\rho-1\bigg)+1= 2\bigg(\bigg\lfloor\frac{\mu}{2}\bigg\rfloor+\beta-1\bigg)+1\leq 2\bigg(\frac{\mu}{2}+\beta-1\bigg)+1=2\beta+\mu-1\leq \delta-1$$
by \eqref{eq3}, the vertex sets described by \eqref{V1i}-\eqref{U2i} are well-defined. Let $S$ be the set of vertices that are not belong
to any of the sets described by \eqref{V1i}-\eqref{U2i}.
Since $\lceil\frac{\beta}{2}\rceil+\lfloor\frac{\beta}{2}\rfloor=\beta$ and $\lfloor\frac{\mu+1}{2}\rfloor+\lfloor\frac{\mu}{2}\rfloor=\mu$,
$$|S|=|G|-4\bigg\lceil\frac{\beta}{2}\bigg\rceil-2\bigg( \bigg\lfloor\frac{\mu+1}{2}\bigg\rfloor-\rho\bigg)-4 \bigg\lfloor\frac{\beta}{2}\bigg\rfloor-2\bigg( \bigg\lfloor\frac{\mu}{2}\bigg\rfloor+\rho\bigg)
=|G|-4\beta-2\mu=\mu.$$
Let $S=\bigg\{z^1_i~\bigg|~2\bigg\lceil\frac{\beta}{2}\bigg\rceil\leq i\leq 2\bigg\lceil\frac{\beta}{2}\bigg\rceil+\bigg\lfloor\frac{\mu+1}{2}\bigg\rfloor-\rho-1\bigg\} \bigcup
\bigg\{z^2_i~\bigg|~2\bigg\lfloor\frac{\beta}{2}\bigg\rfloor\leq i\leq 2\bigg\lfloor\frac{\beta}{2}\bigg\rfloor+\bigg\lfloor\frac{\mu}{2}\bigg\rfloor+\rho-1\bigg\}$ and let
\begin{align*}
  W^1_i&=U^1_i\cup \{z^1_i\}, 2\bigg\lceil\frac{\beta}{2}\bigg\rceil\leq i\leq 2\bigg\lceil\frac{\beta}{2}\bigg\rceil+\bigg\lfloor\frac{\mu+1}{2}\bigg\rfloor-\rho-1\\
  W^2_i&=U^2_i\cup \{z^2_i\}, 2\bigg\lfloor\frac{\beta}{2}\bigg\rfloor\leq i\leq 2\bigg\lfloor\frac{\beta}{2}\bigg\rfloor+\bigg\lfloor\frac{\mu}{2}\bigg\rfloor+\rho-1.
\end{align*}
Since the graph induced by $V^1_i$ or $V^2_i$ or $W^1_i$ or $W^2_i$ induce a linear forest in $G$,
$$V^1_1, \cdots, V^1_{\lceil\beta/2\rceil},V^2_1, \cdots, V^2_{\lfloor\beta/2\rfloor},W^1_{2\lceil\beta/2\rceil},\cdots,W^1_{2\lceil\beta/2\rceil+\lfloor(\mu+1)/2\rfloor-\rho-1},W^2_{2\lfloor\beta/2\rfloor},\cdots,
W^2_{2\lfloor\beta/2\rfloor+\lfloor\mu/2\rfloor+\rho-1}$$ is the desired partition of $V(G)$. Note that there are exactly  $\lceil\frac{\beta}{2}\rceil+\lfloor\frac{\beta}{2}\rfloor+(\lfloor\frac{\mu+1}{2}\rfloor-\rho)+(\lfloor\frac{\mu}{2}\rfloor+\rho)=\beta+\mu=k$ subsets in this partition.

\section{A slightly stronger result}\label{sec:3}

In this section, we give a slightly stronger result than Theorem \ref{ELVA}. To begin with, we prove the following lemma.

\begin{lem}\label{lem-smallmaxdegree}
  If $G$ is a graph with $\Delta(G)<\frac{|G|-1}{2}$ and $k\geq\big\lceil\frac{|G|}{4}\big\rceil$ is an integer, then $V(G)$ can be equitably partitioned into $k$ subsets so that each of them induces a linear forest.
\end{lem}

\begin{proof}
First of all, we notice that
\begin{align*}
\Bigg\lceil\frac{|G|}{k}\Bigg\rceil \leq  \Bigg\lceil\frac{|G|}{\big\lceil\frac{|G|}{4}\big\rceil}\Bigg\rceil\leq 4.
\end{align*}
Since
\begin{align*}
\delta(G^c)=|G|-1-\Delta(G)>\frac{|G|-1}{2}
\end{align*}
and $\delta(G^c)$ is an integer, we conclude
\begin{align*}
\delta(G^c)\geq \frac{|G|}{2}=\frac{|G^c|}{2},
\end{align*}
which implies by the well-known Dirac's Theorem that $G^c$ contains a hamiltonian cycle $C$ (note that we then have $|C|=|G^c|=|G|$). Clearly, we can split $C$ into $k$ vertex-disjoint subpaths on three or four vertices if $k\leq\frac{|G|}{3}$, or on two or three vertices if $\frac{|G|}{3}<k\leq\frac{|G|}{2}$, or on one or two vertices if $k>\frac{|G|}{2}$. In each of the above three cases, the vertices of any of the $k$ subpaths induce a linear forest in $G$. This just proves the theorem.
\end{proof}

Combining Theorem \ref{EVAT} with Lemma \ref{lem-smallmaxdegree}, we conclude the following result towards the Equitable Vertex Arboricity Conjecture.

\begin{thm}
For every graph $G$, $V(G)$ can be equitably partitioned into $k$ subsets so that each of them induces a linear forest whenever $k\geq \max\{\big\lceil\frac{\Delta(G)+1}{2}\big\rceil,\big\lceil\frac{|G|}{4}\big\rceil\}$, i.e., $$va^{\equiv}(G)\leq lva^{\equiv}(G)\leq \max\bigg\{\bigg\lceil\frac{\Delta(G)+1}{2}\bigg\rceil,\bigg\lceil\frac{|G|}{4}\bigg\rceil\bigg\}.$$
\end{thm}

\begin{proof}
If $\Delta(G)\geq \frac{|G|-1}{2}$, then $k\geq \max\{\big\lceil\frac{\Delta(G)+1}{2}\big\rceil,\big\lceil\frac{|G|}{4}\big\rceil\}=\big\lceil\frac{\Delta(G)+1}{2}\big\rceil$. By Theorem \ref{EVAT}, we can construct an equitable partition of $V(G)$
into $k$ subsets so that each of them induces a linear forest. If $\Delta(G)<\frac{|G|-1}{2}$, then $k\geq \max\{\big\lceil\frac{\Delta(G)+1}{2}\big\rceil,\big\lceil\frac{|G|}{4}\big\rceil\}=\big\lceil\frac{|G|}{4}\big\rceil$ and $V(G)$ can be equitably partitioned into $k$ subsets so that each of them induces a linear forest
by Lemma \ref{lem-smallmaxdegree}.
\end{proof}

\section*{Acknowledgements}

We are particularly grateful to Weichan Liu who suggests the constructive proofs of Lemmas \ref{lem-exe}--\ref{lem1}, and also thanks Jingfen Lan, Bi Li, Yan Li and Qingsong Zou for their helpful discussions on shortening the proof of Theorem \ref{EVAT}.

\bibliographystyle{srtnumbered}
\bibliography{mybib}

\end{document}